\documentclass{article}
\usepackage{amssymb}
\usepackage{amsfonts}
\usepackage{amsmath}

\newtheorem{theorem}{Theorem}

\newtheorem{definition}[theorem]{Definition}

\newtheorem{notation}[theorem]{Notation}

\newtheorem{remark}[theorem]{Remark}

\newenvironment{proof}[1][Proof]{\noindent\textbf{#1.} }{\ \rule{0.5em}{0.5em}}
\input{tcilatex}
\begin{document}

\title{Homogenization of the heat equation with a vanishing volumetric heat
capacity}
\author{Tatiana Danielsson, Pernilla Johnsen}
\maketitle

\begin{abstract}
This paper is devoted to the homogenization of the heat conduction equation,
with a homogeneous Dirichlet boundary condition, having a periodically
oscillating thermal conductivity and a vanishing volumetric heat capacity. A
homogenization result is established by using the evolution settings of
multiscale and very weak multiscale convergence. In particular, we
investigate how the relation between the volumetric heat capacity and the
microscopic structure effects the homogenized problem and its associated
local problem. It turns out that the properties of the microscopic geometry
of the problem give rise to certain special effects in the homogenization
result.
\end{abstract}

\section{Introduction\label{Introduction}}

By means of periodic homogenization, we study the heat conduction equation
with homogeneous Dirichlet boundary condition. Homogenization is a technique
for mathematically investigating heterogeneous materials like e.g. composite
materials and porous media. Thinking of the material contained in the domain
as having periodically distributed heterogeneities where the period depends
on a parameter $\varepsilon $, we study the limit process as $\varepsilon $
tends to zero.

We study the linear parabolic equation%
\begin{eqnarray}
\varepsilon ^{q}\partial _{t}u_{\varepsilon }\left( x,t\right) -\nabla \cdot
\left( a\left( \frac{x}{\varepsilon },\frac{t}{\varepsilon ^{r}}\right)
\nabla u_{\varepsilon }\left( x,t\right) \right) &=&f(x,t)\text{ in }\Omega
_{T}\text{,}  \notag \\
u_{\varepsilon }\left( x,0\right) &=&u_{0}\left( x\right) \text{ in }\Omega 
\text{,}  \label{ekvation} \\
u_{\varepsilon }\left( x,t\right) &=&0\text{ on }\partial \Omega \times
\left( 0,T\right) \text{,}  \notag
\end{eqnarray}%
where $0<q<r$ are real numbers, $f\in L^{2}(\Omega _{T})$ and $u_{0}\in
L^{2}(\Omega )$. We denote the domain by $\Omega _{T}=\Omega \times \left(
0,T\right) $, where $\Omega \subset 
\mathbb{R}
^{N}$ is open and bounded with smooth boundary $\partial \Omega $ and $%
\left( 0,T\right) \subset 
\mathbb{R}
$ is an open bounded interval. Here, the thermal conductivity is
characterized by the function $a$, which is periodic with respect to the
unit cube $Y$ in $%
\mathbb{R}
^{N}$ in its first variable and with respect to the unit interval $S$ in $%
\mathbb{R}
$ in its second. The coefficient $\varepsilon ^{q}$ in front of the time
derivative represents the volumetric heat capacity. A more detailed
description of the equation will be given in Section \ref{Homogenization}.

As $\varepsilon $ tends to zero, we search for a weak limit $u$ to the
sequence of solutions $\left\{ u_{\varepsilon }\right\} $, where $u$ is the
solution to a so-called homogenized problem, which is characterized by a
local problem. It turns out that equation (\ref{ekvation}) has two special
features. The first one is that the homogenized problem is elliptic, for all 
$0<q<r$, even though the original problem is parabolic. The second feature
is that we have what we refer to as resonance, i.e. a parabolic local
problem, for a different matching between the microscopic scales than the
usual one. In \cite{BLP} it was shown that parabolic equations usually have
resonance if the temporal scale is the square of the spatial one. Several
other studies of parabolic equations, both equations where the coefficient
in front of the time derivative is identical to one and equations with
oscillating coefficients, show resonance for the same type of matching, see
e.g. \cite{Hol}, \cite{NgWo}, \cite{AlPi}, \cite{FHOLstrange-term}, \cite%
{FHOLPmismatch}, \cite{SvWo}, \cite{FHOLParbitrary} and \cite{DoWo}. As we
will see in the homogenization result, equation (\ref{ekvation}) will have
resonance if $r=q+2$, i.e. the matching that gives a parabolic local problem
is not when the temporal scale is equal to the spatial one.

In the homogenization procedure we use evolution settings of multiscale and
very weak multiscale convergence. A gradient characterization and a very
weak multiscale convergence compactness result for sequences bounded in $%
W^{1,2}(0,T;$\newline
$H_{0}^{1}(\Omega ),L^{2}(\Omega ))$, meaning that $\left\{ u_{\varepsilon
}\right\} $ is bounded in $L^{2}(0,T;H_{0}^{1}(\Omega ))$ and $\left\{
\partial _{t}u_{\varepsilon }\right\} $ is bounded in $L^{2}(0,T;H^{-1}(%
\Omega ))$, can be found in e.g. \cite{FHOLParbitrary}. Here, we use a
different approach where the boundedness of the time derivative is replaced
by a certain condition, see (\ref{villkor}) in Theorem \ref%
{Theorem-gradsplit}. This approach was, up to the authors' knowledge, first
used in \cite{Lob} where compactness results for sequences defined on
perforated domains were given. The corresponding compactness results for
sequences defined on non-perforated domains, which we present in Theorem \ref%
{Theorem-gradsplit} and Theorem \ref{Theorem-very weak}, are stated and
proven in \cite{JoLo}. The present paper is a further development of the
work in \cite{JoLo}, where a homogenization result for equation (\ref%
{ekvation}) for the case when $q=1$ and $r=3$ was established.

\begin{notation}
We use the notation $\mathcal{Y}_{n,m}=Y^{n}\times S^{m}$ with $%
Y^{n}=Y_{1}\times Y_{2}\times \cdots \times Y_{n}$ and $S^{m}=S_{1}\times
S_{2}\times \cdots \times S_{m}$, where $Y_{1}=Y_{2}=\ldots =Y_{n}=Y=\left(
0,1\right) ^{N}$ and $S_{1}=S_{2}=\ldots =S_{m}=S=\left( 0,1\right) $.
Further, we let $y^{n}=y_{1},y_{2},\ldots ,y_{n}$, $dy^{n}=dy_{1}dy_{2}%
\cdots dy_{n}$, $s^{m}=s_{1},s_{2},\ldots ,s_{m}$ and $ds^{m}=ds_{1}ds_{2}%
\cdots ds_{m}$. Moreover, we denote by $W^{1,2}(0,T;H_{0}^{1}(\Omega
),L^{2}(\Omega ))$ the space of all functions $u\in
L^{2}(0,T;H_{0}^{1}(\Omega ))$ such that $\partial _{t}u\in
L^{2}(0,T;H^{-1}(\Omega ))$. The subscript $_{\sharp }$ is used on function
spaces to denote periodicity of the functions involved over the domain in
question. Lastly, for $k=1,\ldots ,n$ and $j=1,\ldots ,m$, the scale
functions $\varepsilon _{k}(\varepsilon )$ and $\varepsilon _{j}^{\prime
}(\varepsilon )$ are strictly positive and tend to zero as $\varepsilon $
does and we denote lists of spatial and temporal scales by $\left\{
\varepsilon _{1},\ldots ,\varepsilon _{n}\right\} $ and $\left\{ \varepsilon
_{1}^{\prime },\ldots ,\varepsilon _{m}^{\prime }\right\} $, respectively.
\end{notation}

\section{Preliminaries\label{Preliminaries}}

Our main tools, evolution multiscale and very weak evolution multiscale
convergence, are generalizations and modifications of the classical
two-scale convergence. A sequence $\left\{ u_{\varepsilon }\right\} $ in $%
L^{2}(\Omega )$ is said to two-scale converge to $u_{0}\in L^{2}(\Omega
\times Y)$ if 
\begin{equation*}
\underset{\varepsilon \rightarrow 0}{\lim }\int_{\Omega }u_{\varepsilon
}\left( x\right) v\left( x,\frac{x}{\varepsilon }\right) dx=\int_{\Omega
}\int_{Y}u_{0}\left( x,y\right) v\left( x,y\right) dydx
\end{equation*}%
for all $v\in L^{2}(\Omega ;C_{\sharp }(Y))$.

Two-scale convergence was introduced by Nguetseng in \cite{Ngu1} and \cite%
{Ngu2}, where he uses the concept to homogenize a linear elliptic problem
with one microscopic spatial scale. In \cite{Al1}, Allaire gives a
compactness result for a different class of test functions and applies the
concept to e.g. nonlinear elliptic problems and problems on perforated
domains.\ The generalization of two-scale convergence to sequences with
multiple microscopic scales in space was provided by Allaire and Briane in 
\cite{AlBr}, where they give the definition and a compactness result for the
concept. Following \cite{AlBr}, we say that a sequence $\left\{
u_{\varepsilon }\right\} $ $\left( n+1\right) $-scale converges to $u_{0}\in
L^{2}(\Omega \times Y^{n})$ if%
\begin{equation*}
\underset{\varepsilon \rightarrow 0}{\lim }\int_{\Omega }u_{\varepsilon
}\left( x\right) v\left( x,\frac{x}{\varepsilon _{1}},\cdots ,\frac{x}{%
\varepsilon _{n}}\right) dx=\int_{\Omega }\int_{Y^{n}}u_{0}\left(
x,y^{n}\right) v\left( x,y^{n}\right) dy^{n}dx
\end{equation*}%
for all $v\in L^{2}(\Omega ;C_{\sharp }(Y^{n}))$.

In \cite{PerLic} (see also the appendix of \cite{FHOLParbitrary}),
compactness results were given for an arbitrary number of scales in both
space and time, extending the concept of multiscale convergence to an
analogous evolution setting.

\begin{definition}[Evolution multiscale convergence]
A sequence $\left\{ u_{\varepsilon }\right\} $ in\newline
$L^{2}(\Omega _{T})$ is said to $\left( n+1,m+1\right) $-scale converge to $%
u_{0}\in L^{2}(\Omega _{T}\times \mathcal{Y}_{n,m})$ if%
\begin{gather*}
\lim_{\varepsilon \rightarrow 0}\int_{\Omega _{T}}u_{\varepsilon }\left(
x,t\right) v\left( x,t,\frac{x}{\varepsilon _{1}},\cdots ,\frac{x}{%
\varepsilon _{n}},\frac{t}{\varepsilon _{1}^{\prime }},\cdots ,\frac{t}{%
\varepsilon _{m}^{\prime }}\right) dxdt \\
=\int_{\Omega _{T}}\int_{\mathcal{Y}_{n,m}}u_{0}\left(
x,t,y^{n},s^{m}\right) v\left( x,t,y^{n},s^{m}\right) dy^{n}ds^{m}dxdt
\end{gather*}%
for all $v\in L^{2}(\Omega _{T};C_{\sharp }(\mathcal{Y}_{n,m})).$ This is
denoted by 
\begin{equation*}
u_{\varepsilon }\left( x,t\right) \overset{n+1,m+1}{\rightharpoonup }%
u_{0}\left( x,t,y^{n},s^{m}\right) \text{.}
\end{equation*}
\end{definition}

Before we proceed with the compactness results we make some additional
assumptions on the microscopic scales. The scales in a list are said to be
separated if%
\begin{equation*}
\lim_{\varepsilon \rightarrow 0}\frac{\varepsilon _{k+1}}{\varepsilon _{k}}=0
\end{equation*}%
and well-separated if there exists a positive integer $\ell $ such that%
\begin{equation*}
\lim_{\varepsilon \rightarrow 0}\frac{1}{\varepsilon _{k}}\left( \frac{%
\varepsilon _{k+1}}{\varepsilon _{k}}\right) ^{\ell }=0\text{,}
\end{equation*}%
where $k=1,\ldots ,n-1$. When having two lists of microscopic scales, one
spatial and one temporal, we have the following generalization of
separatedness and well-separatedness, the concept of jointly
(well-)separatedness. The definition was first given by Persson, see e.g. 
\cite{PerMonotone} where a more technically formulated version is given.

\begin{definition}[Jointly (well-)separated scales]
Let $\left\{ \varepsilon _{1},\ldots ,\varepsilon _{n}\right\} $ and $%
\left\{ \varepsilon _{1}^{\prime },\ldots ,\varepsilon _{m}^{\prime
}\right\} $ be lists of (well-)separated scales. Collect all elements from
both lists in one common list. If from possible duplicates, where by
duplicates we mean scales which tend to zero equally fast, one member of
each such pair is removed and the list in order of magnitude of all the
remaining elements is (well-)separated, the lists $\left\{ \varepsilon
_{1},\ldots ,\varepsilon _{n}\right\} $ and $\left\{ \varepsilon
_{1}^{\prime },\ldots ,\varepsilon _{m}^{\prime }\right\} $ are said to be
jointly (well-)separated.
\end{definition}

A compactness result for $\left( n+1,m+1\right) $-scale convergence reads as
follows.

\begin{theorem}
\label{Theorem-begr ger multiskala}Let $\left\{ u_{\varepsilon }\right\} $
be a bounded sequence in $L^{2}(\Omega _{T})$ and suppose that the lists $%
\left\{ \varepsilon _{1},\ldots ,\varepsilon _{n}\right\} $ and $\left\{
\varepsilon _{1}^{\prime },\ldots ,\varepsilon _{m}^{\prime }\right\} $ are
jointly separated. Then, up to a subsequence,%
\begin{equation*}
u_{\varepsilon }\left( x,t\right) \overset{n+1,m+1}{\rightharpoonup }%
u_{0}\left( x,t,y^{n},s^{m}\right) \text{,}
\end{equation*}%
where $u_{0}\in L^{2}(\Omega _{T}\times \mathcal{Y}_{n,m})$.
\end{theorem}

\begin{proof}
See Theorem A.1 in \cite{FHOLParbitrary}.
\end{proof}

The following gradient characterization, which is adapted to our problem,
will be important in the homogenization of (\ref{ekvation}).

\begin{theorem}
\label{Theorem-gradsplit}Assume that $\left\{ u_{\varepsilon }\right\} $ is
bounded in $L^{2}(0,T;H_{0}^{1}(\Omega ))$ and, for any $v_{1}\in D(\Omega )$%
, $c_{1}\in D(0,T)$, $c_{2}\in C_{\sharp }^{\infty }(S)$ and $r>0$,%
\begin{equation}
\lim_{\varepsilon \rightarrow 0}\int_{\Omega _{T}}u_{\varepsilon }\left(
x,t\right) v_{1}\left( x\right) \partial _{t}\left( \varepsilon
^{r}c_{1}\left( t\right) c_{2}\left( \frac{t}{\varepsilon ^{r}}\right)
\right) dxdt=0\text{.}  \label{villkor}
\end{equation}%
Then, for $n=m=1$ with $\varepsilon _{1}=\varepsilon $ and $\varepsilon
_{1}^{\prime }=\varepsilon ^{r}$, up to a subsequence,%
\begin{equation*}
u_{\varepsilon }\left( x,t\right) \rightharpoonup u\left( x,t\right) \text{
in }L^{2}(0,T;H_{0}^{1}(\Omega ))
\end{equation*}%
and%
\begin{equation}
\nabla u_{\varepsilon }\left( x,t\right) \overset{2,2}{\rightharpoonup }%
\nabla u\left( x,t\right) +\nabla _{y}u_{1}\left( x,t,y,s\right) \text{,}
\label{gradsplit}
\end{equation}%
where $u\in L^{2}(0,T;H_{0}^{1}(\Omega ))$ and $u_{1}\in L^{2}(\Omega
_{T}\times S;H_{\sharp }^{1}(Y))$.
\end{theorem}

\begin{proof}
See Theorem 2.7 in \cite{JoLo}.
\end{proof}

Evolution multiscale convergence solely is not enough to handle certain
sequences appearing in the homogenization of (\ref{ekvation}). Therefore, we
introduce very weak evolution multiscale convergence. This type of
convergence originates from \cite{Hol}, where it is used to obtain
homogenization and corrector results for linear parabolic problems with one
microscopic scale in space and time respectively. In \cite{NgWo}, further
progress in the context of $\Sigma $-convergence led to a closely related
result and a simplification for the applicability in the homogenization
procedure. The present form of the concept was given for an arbitrary number
of spatial scales in \cite{FHOPvery-weak}, where also the name "very weak
multiscale convergence" was established. Following \cite{PerPhD} and \cite%
{FHOLParbitrary}, we give the evolution version of very weak multiscale
convergence including arbitrarily many spatial and temporal scales.

\begin{definition}[Very weak evolution multiscale convergence]
A sequence $\left\{ w_{\varepsilon }\right\} $ in $L^{1}(\Omega _{T})$ is
said to $(n+1,m+1)$-scale converge very weakly to $w_{0}\in L^{1}(\Omega
_{T}\times \mathcal{Y}_{n,m})$ if%
\begin{gather*}
\lim_{\varepsilon \rightarrow 0}\int_{\Omega _{T}}w_{\varepsilon }\left(
x,t\right) v_{1}\left( x,\frac{x}{\varepsilon _{1}},\ldots ,\frac{x}{%
\varepsilon _{n-1}}\right) v_{2}\left( \frac{x}{\varepsilon _{n}}\right)
c\left( t,\frac{t}{\varepsilon _{1}^{\prime }},\ldots ,\frac{t}{\varepsilon
_{m}^{\prime }}\right) dxdt \\
=\int_{\Omega _{T}}\int_{\mathcal{Y}_{n,m}}w_{0}\left(
x,t,y^{n},s^{m}\right) v_{1}(x,y^{n-1})v_{2}\left( y_{n}\right)
c(t,s^{m})dy^{n}ds^{m}dxdt
\end{gather*}%
for any $v_{1}\in D(\Omega ;C_{\sharp }^{\infty }(Y^{n-1})),$ $v_{2}\in
C_{\sharp }^{\infty }(Y_{n})/%
\mathbb{R}
$ and $c\in D(0,T;C_{\sharp }^{\infty }(S^{m}))$, where 
\begin{equation}
\int_{Y_{n}}w_{0}\left( x,t,y^{n},s^{m}\right) dy_{n}=0\text{.}
\label{very weak unik}
\end{equation}%
We write%
\begin{equation*}
w_{\varepsilon }\left( x,t\right) \underset{vw}{\overset{n+1,m+1}{%
\rightharpoonup }}w_{0}\left( x,t,y^{n},s^{m}\right) \text{.}
\end{equation*}
\end{definition}

\begin{remark}
Due to (\ref{very weak unik}) the limit is unique.
\end{remark}

We now give a compactness result for very weak evolution multiscale
convergence which will play a vital role, complementing Theorem \ref%
{Theorem-gradsplit}, in the homogenization of (\ref{ekvation}). Note that (%
\ref{villkor2}) in the theorem below is the same as (\ref{villkor}) in
Theorem \ref{Theorem-gradsplit}.

\begin{theorem}
\label{Theorem-very weak}Assume that $\left\{ u_{\varepsilon }\right\} $ is
bounded in $L^{2}(0,T;H_{0}^{1}(\Omega ))$ and, for any $v_{1}\in D(\Omega )$%
, $c_{1}\in D(0,T)$, $c_{2}\in C_{\sharp }^{\infty }(S)$ and $r>0$,%
\begin{equation}
\lim_{\varepsilon \rightarrow 0}\int_{\Omega _{T}}u_{\varepsilon }\left(
x,t\right) v_{1}\left( x\right) \partial _{t}\left( \varepsilon
^{r}c_{1}\left( t\right) c_{2}\left( \frac{t}{\varepsilon ^{r}}\right)
\right) dxdt=0\text{.}  \label{villkor2}
\end{equation}%
Then, for $n=m=1$ with $\varepsilon _{1}=\varepsilon $ and $\varepsilon
_{1}^{\prime }=\varepsilon ^{r}$, up to a subsequence,%
\begin{equation*}
\varepsilon ^{-1}u_{\varepsilon }\left( x,t\right) \overset{2,2}{\underset{vw%
}{\rightharpoonup }}u_{1}\left( x,t,y,s\right) \text{,}
\end{equation*}%
where $u_{1}\in L^{2}(\Omega _{T}\times S;H_{\sharp }^{1}(Y)/%
\mathbb{R}
)$ is the same as in (\ref{gradsplit}) in Theorem \ref{Theorem-gradsplit}.
\end{theorem}

\begin{proof}
See Theorem 2.10 in \cite{JoLo}.
\end{proof}

\section{Homogenization\label{Homogenization}}

Let us now investigate problem (\ref{ekvation}), i.e. establish a
homogenization result for the equation%
\begin{eqnarray}
\varepsilon ^{q}\partial _{t}u_{\varepsilon }\left( x,t\right) -\nabla \cdot
\left( a\left( \frac{x}{\varepsilon },\frac{t}{\varepsilon ^{r}}\right)
\nabla u_{\varepsilon }\left( x,t\right) \right) &=&f(x,t)\text{ in }\Omega
_{T}\text{,}  \notag \\
u_{\varepsilon }\left( x,0\right) &=&u_{0}\left( x\right) \text{ in }\Omega 
\text{,}  \label{ekvation2} \\
u_{\varepsilon }\left( x,t\right) &=&0\text{ on }\partial \Omega \times
\left( 0,T\right) \text{,}  \notag
\end{eqnarray}%
where $0<q<r$, $f\in L^{2}(\Omega _{T})$ and $u_{0}\in L^{2}(\Omega )$. The
coefficient $a\in C_{\sharp }(\mathcal{Y}_{1,1})^{N\times N}$ satisfies the
coercivity condition%
\begin{equation}
a\left( y,s\right) \xi \cdot \xi \geq C_{0}\left\vert \xi \right\vert ^{2}
\label{koercivitet}
\end{equation}%
for a.e. $\left( y,s\right) \in \mathcal{Y}_{1,1}$, for every $\xi \in 
\mathbb{R}
^{N}$ and for some $C_{0}>0$. According to Section 23.7 in \cite{ZeiIIA} the
problem possesses a unique solution. The weak form of (\ref{ekvation2}) is%
\begin{gather}
\int_{\Omega _{T}}-\varepsilon ^{q}u_{\varepsilon }\left( x,t\right) v\left(
x\right) \partial _{t}c\left( t\right) +a\left( \frac{x}{\varepsilon },\frac{%
t}{\varepsilon ^{r}}\right) \nabla u_{\varepsilon }\left( x,t\right) \cdot
\nabla v\left( x\right) c\left( t\right) dxdt  \label{svag form} \\
=\int_{\Omega _{T}}f\left( x,t\right) v\left( x\right) c\left( t\right) dxdt%
\text{,}  \notag
\end{gather}%
for all $v\in H_{0}^{1}(\Omega )$ and $c\in D(0,T)$.

We will now show that the solution to (\ref{ekvation2}) is bounded in $%
L^{2}(0,T;H_{0}^{1}(\Omega ))$, i.e. it satisfies the a priori estimate%
\begin{equation}
\left\Vert u_{\varepsilon }\right\Vert _{L^{2}(0,T;H_{0}^{1}(\Omega ))}\leq C%
\text{,}  \label{a priori villkor}
\end{equation}%
where $C>0$ is a constant independent of $\varepsilon $. By Section 30.3 in 
\cite{ZeiIIB}, using $u_{\varepsilon }\in W^{1,2}(0,T;H_{0}^{1}(\Omega
),L^{2}(\Omega ))$ as a test function, the operator form of (\ref{ekvation2}%
) is%
\begin{gather*}
\int_{0}^{T}\varepsilon ^{q}\left\langle \partial _{t}u_{\varepsilon
},u_{\varepsilon }\right\rangle _{H^{-1}(\Omega ),H_{0}^{1}(\Omega
)}dt+\int_{\Omega _{T}}a\left( \frac{x}{\varepsilon },\frac{t}{\varepsilon
^{r}}\right) \nabla u_{\varepsilon }\left( x,t\right) \cdot \nabla
u_{\varepsilon }\left( x,t\right) dxdt \\
=\int_{\Omega _{T}}f\left( x,t\right) u_{\varepsilon }\left( x,t\right) dxdt%
\text{.}
\end{gather*}%
Multiplying by 2 and using the integration by parts formula (25) from
Section 23.6 in \cite{ZeiIIA} we obtain%
\begin{gather*}
\int_{\Omega }\varepsilon ^{q}\left( \left( u_{\varepsilon }\left(
x,T\right) \right) ^{2}-\left( u_{0}\left( x\right) \right) ^{2}\right)
dx+2\int_{\Omega _{T}}a\left( \frac{x}{\varepsilon },\frac{t}{\varepsilon
^{r}}\right) \nabla u_{\varepsilon }\left( x,t\right) \cdot \nabla
u_{\varepsilon }\left( x,t\right) dxdt \\
=2\int_{\Omega _{T}}f\left( x,t\right) u_{\varepsilon }\left( x,t\right) dxdt
\end{gather*}%
or, equivalently,%
\begin{gather*}
\varepsilon ^{q}\left\Vert u_{\varepsilon }\left( \cdot ,T\right)
\right\Vert _{L^{2}(\Omega )}^{2}+2\int_{\Omega _{T}}a\left( \frac{x}{%
\varepsilon },\frac{t}{\varepsilon ^{r}}\right) \nabla u_{\varepsilon
}\left( x,t\right) \cdot \nabla u_{\varepsilon }\left( x,t\right) dxdt \\
=\varepsilon ^{q}\left\Vert u_{0}\right\Vert _{L^{2}(\Omega
)}^{2}+2\int_{\Omega _{T}}f\left( x,t\right) u_{\varepsilon }\left(
x,t\right) dxdt\text{.}
\end{gather*}%
The coercivity condition (\ref{koercivitet}) states that%
\begin{gather*}
2\int_{\Omega _{T}}a\left( \frac{x}{\varepsilon },\frac{t}{\varepsilon ^{r}}%
\right) \nabla u_{\varepsilon }\left( x,t\right) \cdot \nabla u_{\varepsilon
}\left( x,t\right) dxdt \\
\geq 2C_{0}\int_{\Omega _{T}}\left\vert \nabla u_{\varepsilon }\right\vert
^{2}dxdt=2C_{0}\left\Vert u_{\varepsilon }\right\Vert
_{L^{2}(0,T;H_{0}^{1}(\Omega ))}^{2}\text{,}
\end{gather*}%
which gives us%
\begin{gather}
\varepsilon ^{q}\left\Vert u_{\varepsilon }\left( \cdot ,T\right)
\right\Vert _{L^{2}(\Omega )}^{2}+2C_{0}\left\Vert u_{\varepsilon
}\right\Vert _{L^{2}(0,T;H_{0}^{1}(\Omega ))}^{2}  \label{bevis1} \\
\leq \varepsilon ^{q}\left\Vert u_{0}\right\Vert _{L^{2}(\Omega
)}^{2}+2\int_{\Omega _{T}}f\left( x,t\right) u_{\varepsilon }\left(
x,t\right) dxdt\text{.}  \notag
\end{gather}%
Using the\ Poincar\'{e} inequality%
\begin{equation*}
\left\Vert u_{\varepsilon }\right\Vert _{L^{2}(\Omega _{T})}^{2}\leq
C_{2}\left\Vert u_{\varepsilon }\right\Vert _{L^{2}(0,T;H_{0}^{1}(\Omega
))}^{2}\text{,}
\end{equation*}%
where $C_{2}>0$ depends only on $\Omega $, and the elementary inequality%
\begin{equation*}
2xy\leq C_{1}x^{2}+C_{1}^{-1}y^{2}\text{,}
\end{equation*}%
with $C_{1}=C_{0}^{-1}C_{2}$, we have%
\begin{gather*}
2\int_{\Omega _{T}}f\left( x,t\right) u_{\varepsilon }\left( x,t\right)
dxdt\leq C_{0}^{-1}C_{2}\left\Vert f\right\Vert _{L^{2}(\Omega
_{T})}^{2}+\left( C_{0}^{-1}C_{2}\right) ^{-1}\left\Vert u_{\varepsilon
}\right\Vert _{L^{2}\left( \Omega _{T}\right) }^{2} \\
\leq C_{0}^{-1}C_{2}\left\Vert f\right\Vert _{L^{2}(\Omega
_{T})}^{2}+C_{0}C_{2}^{-1}C_{2}\left\Vert u_{\varepsilon }\right\Vert
_{L^{2}(0,T;H_{0}^{1}(\Omega ))}^{2} \\
=C_{0}^{-1}C_{2}\left\Vert f\right\Vert _{L^{2}(\Omega
_{T})}^{2}+C_{0}\left\Vert u_{\varepsilon }\right\Vert
_{L^{2}(0,T;H_{0}^{1}(\Omega ))}^{2}\text{.}
\end{gather*}%
Now (\ref{bevis1}) becomes%
\begin{gather*}
\varepsilon ^{q}\left\Vert u_{\varepsilon }\left( \cdot ,T\right)
\right\Vert _{L^{2}(\Omega )}^{2}+2C_{0}\left\Vert u_{\varepsilon
}\right\Vert _{L^{2}(0,T;H_{0}^{1}(\Omega ))}^{2} \\
\leq \varepsilon ^{q}\left\Vert u_{0}\right\Vert _{L^{2}(\Omega
)}^{2}+C_{0}^{-1}C_{2}\left\Vert f\right\Vert _{L^{2}(\Omega
_{T})}^{2}+C_{0}\left\Vert u_{\varepsilon }\right\Vert
_{L^{2}(0,T;H_{0}^{1}(\Omega ))}^{2}
\end{gather*}%
or, rewriting,%
\begin{equation*}
\left\Vert u_{\varepsilon }\right\Vert _{L^{2}(0,T;H_{0}^{1}(\Omega
))}^{2}\leq \varepsilon ^{q}C_{0}^{-1}\left\Vert u_{0}\right\Vert
_{L^{2}(\Omega )}^{2}+C_{0}^{-2}C_{2}\left\Vert f\right\Vert _{L^{2}(\Omega
_{T})}^{2}-\varepsilon ^{q}C_{0}^{-1}\left\Vert u_{\varepsilon }\left( \cdot
,T\right) \right\Vert _{L^{2}(\Omega )}^{2}\text{.}
\end{equation*}%
Noting that%
\begin{equation*}
\varepsilon ^{q}C_{0}^{-1}\left\Vert u_{\varepsilon }\left( \cdot ,T\right)
\right\Vert _{L^{2}(\Omega )}^{2}\geq 0
\end{equation*}%
we arrive at 
\begin{equation*}
\left\Vert u_{\varepsilon }\right\Vert _{L^{2}(0,T;H_{0}^{1}(\Omega
))}^{2}\leq \varepsilon ^{q}C_{0}^{-1}\left\Vert u_{0}\right\Vert
_{L^{2}(\Omega )}^{2}+C_{0}^{-2}C_{2}\left\Vert f\right\Vert _{L^{2}(\Omega
_{T})}^{2}\text{,}
\end{equation*}%
which implies (\ref{a priori villkor}), i.e. we have shown that $\left\{
u_{\varepsilon }\right\} $ is bounded in $L^{2}(0,T;H_{0}^{1}(\Omega ))$.

Before we are ready to give the homogenization result we prove that the
assumption used in Theorems \ref{Theorem-gradsplit} and \ref{Theorem-very
weak} is satisfied, i.e. that%
\begin{equation}
\lim_{\varepsilon \rightarrow 0}\int_{\Omega _{T}}u_{\varepsilon }\left(
x,t\right) v_{1}\left( x\right) \partial _{t}\left( \varepsilon
^{r}c_{1}\left( t\right) c_{2}\left( \frac{t}{\varepsilon ^{r}}\right)
\right) dxdt=0  \label{villkor3}
\end{equation}%
for $v_{1}\in D(\Omega )$, $c_{1}\in D(0,T)$ and $c_{2}\in C_{\sharp
}^{\infty }(S)$ and $r>0$. By using the weak form (\ref{svag form}) with the
test function%
\begin{equation*}
v\left( x\right) c\left( t\right) =\varepsilon ^{r-q}v_{1}\left( x\right)
c_{1}\left( t\right) c_{2}\left( \frac{t}{\varepsilon ^{r}}\right) \text{,}
\end{equation*}%
where $0<q<r$, $v_{1}\in D(\Omega )$, $c_{1}\in D(0,T)$ and $c_{2}\in
C_{\sharp }^{\infty }(S)$, we get%
\begin{gather*}
\int_{\Omega _{T}}-\varepsilon ^{q}u_{\varepsilon }\left( x,t\right)
\varepsilon ^{r-q}v_{1}\left( x\right) \partial _{t}\left( c_{1}\left(
t\right) c_{2}\left( \frac{t}{\varepsilon ^{r}}\right) \right) dxdt \\
+\int_{\Omega _{T}}a\left( \frac{x}{\varepsilon },\frac{t}{\varepsilon ^{r}}%
\right) \nabla u_{\varepsilon }\left( x,t\right) \cdot \varepsilon
^{r-q}\nabla v_{1}\left( x\right) c_{1}\left( t\right) c_{2}\left( \frac{t}{%
\varepsilon ^{r}}\right) dxdt \\
=\int_{\Omega _{T}}f\left( x,t\right) \varepsilon ^{r-q}v_{1}\left( x\right)
c_{1}\left( t\right) c_{2}\left( \frac{t}{\varepsilon ^{r}}\right) dxdt
\end{gather*}%
and by rearranging we obtain%
\begin{gather*}
\int_{\Omega _{T}}u_{\varepsilon }\left( x,t\right) v_{1}\left( x\right)
\partial _{t}\left( \varepsilon ^{r}c_{1}\left( t\right) c_{2}\left( \frac{t%
}{\varepsilon ^{r}}\right) \right) dxdt \\
=\int_{\Omega _{T}}\varepsilon ^{r-q}a\left( \frac{x}{\varepsilon },\frac{t}{%
\varepsilon ^{r}}\right) \nabla u_{\varepsilon }\left( x,t\right) \cdot
\nabla v_{1}\left( x\right) c_{1}\left( t\right) c_{2}\left( \frac{t}{%
\varepsilon ^{r}}\right) dxdt \\
-\int_{\Omega _{T}}\varepsilon ^{r-q}f\left( x,t\right) v_{1}\left( x\right)
c_{1}\left( t\right) c_{2}\left( \frac{t}{\varepsilon ^{r}}\right) dxdt\text{%
.}
\end{gather*}%
From (\ref{a priori villkor}) we know that $\left\{ u_{\varepsilon }\right\} 
$ is bounded in $L^{2}(0,T;H_{0}^{1}(\Omega ))$ and therefore $\left\{
\nabla u_{\varepsilon }\right\} $ is bounded in $L^{2}(\Omega _{T})^{N}$ and
we have%
\begin{gather*}
\lim_{\varepsilon \rightarrow 0}\int_{\Omega _{T}}u_{\varepsilon }\left(
x,t\right) v_{1}\left( x\right) \partial _{t}\left( \varepsilon
^{r}c_{1}\left( t\right) c_{2}\left( \frac{t}{\varepsilon ^{r}}\right)
\right) dxdt \\
=\lim_{\varepsilon \rightarrow 0}\left( \int_{\Omega _{T}}\varepsilon
^{r-q}a\left( \frac{x}{\varepsilon },\frac{t}{\varepsilon ^{r}}\right)
\nabla u_{\varepsilon }\left( x,t\right) \cdot \nabla v_{1}\left( x\right)
c_{1}\left( t\right) c_{2}\left( \frac{t}{\varepsilon ^{r}}\right)
dxdt\right. \\
-\left. \int_{\Omega _{T}}\varepsilon ^{r-q}f\left( x,t\right) v_{1}\left(
x\right) c_{1}\left( t\right) c_{2}\left( \frac{t}{\varepsilon ^{r}}\right)
dxdt\right) =0\text{,}
\end{gather*}%
meaning that (\ref{villkor3}) is fulfilled.

Finally, we give the homogenization result.

\begin{theorem}
\label{Theorem-homogenisering}Let $\left\{ u_{\varepsilon }\right\} $ be a
sequence of solutions to (\ref{ekvation2}) in $W^{1,2}(0,T;H_{0}^{1}(\Omega
),L^{2}(\Omega ))$. Then it holds that%
\begin{equation}
u_{\varepsilon }\left( x,t\right) \rightharpoonup u\left( x,t\right) \text{
in }L^{2}(0,T;H_{0}^{1}(\Omega ))  \label{svag}
\end{equation}%
and%
\begin{equation}
\nabla u_{\varepsilon }\left( x,t\right) \overset{2,2}{\rightharpoonup }%
\nabla u\left( x,t\right) +\nabla _{y}u_{1}\left( x,t,y,s\right) \text{,}
\label{gradsplit2}
\end{equation}%
where $u\in $ $L^{2}(0,T;H_{0}^{1}(\Omega ))$ and $u_{1}\in L^{2}(\Omega
_{T}\times S;H_{\sharp }^{1}(Y)/%
\mathbb{R}
)$. Here, $u$ is the unique solution to the homogenized problem%
\begin{eqnarray}
-\nabla \cdot \left( b\nabla u\left( x,t\right) \right) &=&f\left(
x,t\right) \text{ in }\Omega _{T}\text{,}  \label{homogeniserat problem} \\
u\left( x,t\right) &=&0\text{ on }\partial \Omega \times \left( 0,T\right) 
\notag
\end{eqnarray}%
with%
\begin{equation}
b\nabla u\left( x,t\right) =\int_{\mathcal{Y}_{1,1}}a\left( y,s\right)
\left( \nabla u\left( x,t\right) +\nabla _{y}u_{1}\left( x,t,y,s\right)
\right) dyds\text{.}  \label{koefficient}
\end{equation}%
For $q<r<q+2$, $u_{1}$ is determined by the elliptic local problem%
\begin{equation}
-\nabla _{y}\cdot \left( a\left( y,s\right) \cdot \left( \nabla u\left(
x,t\right) +\nabla _{y}u_{1}\left( x,t,y,s\right) \right) \right) =0
\label{lokalt problem1}
\end{equation}%
and for $r=q+2$ by the parabolic local problem%
\begin{equation}
\partial _{s}u_{1}\left( x,t,y,s\right) -\nabla _{y}\cdot \left( a\left(
y,s\right) \left( \nabla u\left( x,t\right) +\nabla _{y}u_{1}\left(
x,t,y,s\right) \right) \right) =0\text{.}  \label{lokalt problem2}
\end{equation}%
For $r>q+2$, $u_{1}$ is determined by the elliptic local problem%
\begin{equation}
-\nabla _{y}\cdot \left( \left( \int_{S}a\left( y,s\right) ds\right) \cdot
\left( \nabla u\left( x,t\right) +\nabla _{y}u_{1}\left( x,t,y\right)
\right) \right) =0  \label{lokalt problem3}
\end{equation}%
and since $u_{1}$ is independent of $s$, the coefficient (\ref{koefficient})
can, in this case, be expressed as%
\begin{equation*}
b\nabla u\left( x,t\right) =\int_{Y}\left( \int_{S}a\left( y,s\right)
ds\right) \left( \nabla u\left( x,t\right) +\nabla _{y}u_{1}\left(
x,t,y\right) \right) dy\text{.}
\end{equation*}
\end{theorem}

\begin{proof}
Since (\ref{a priori villkor}) and (\ref{villkor3}) are satisfied the
convergences (\ref{svag}) and (\ref{gradsplit2}) holds, according to Theorem %
\ref{Theorem-gradsplit}. To obtain the homogenized problem we choose, in the
weak form (\ref{svag form}), the test function%
\begin{equation*}
v\left( x\right) c\left( t\right) =v_{1}\left( x\right) c_{1}\left( t\right) 
\text{,}
\end{equation*}%
where $v_{1}\in H_{0}^{1}(\Omega )$ and $c_{1}\in D(0,T)$, giving%
\begin{gather*}
\int_{\Omega _{T}}-\varepsilon ^{q}u_{\varepsilon }\left( x,t\right)
v_{1}\left( x\right) \partial _{t}c_{1}\left( t\right) +a\left( \frac{x}{%
\varepsilon },\frac{t}{\varepsilon ^{r}}\right) \nabla u_{\varepsilon
}\left( x,t\right) \cdot \nabla v_{1}\left( x\right) c_{1}\left( t\right)
dxdt \\
=\int_{\Omega _{T}}f\left( x,t\right) v_{1}\left( x\right) c_{1}\left(
t\right) dxdt\text{.}
\end{gather*}%
Letting $\varepsilon $ tend to zero we have%
\begin{gather*}
\int_{\Omega _{T}}\int_{\mathcal{Y}_{1,1}}a\left( y,s\right) \left( \nabla
u\left( x,t\right) +\nabla _{y}u_{1}\left( x,t,y,s\right) \right) \cdot
\nabla v_{1}\left( x\right) c_{1}\left( t\right) dydsdxdt \\
=\int_{\Omega _{T}}f\left( x,t\right) v_{1}\left( x\right) c_{1}\left(
t\right) dxdt\text{,}
\end{gather*}%
and by the Variational lemma we arrive at%
\begin{gather*}
\int_{\Omega }\int_{\mathcal{Y}_{1,1}}a\left( y,s\right) \left( \nabla
u\left( x,t\right) +\nabla _{y}u_{1}\left( x,t,y,s\right) \right) \cdot
\nabla v_{1}\left( x\right) dydsdx \\
=\int_{\Omega }f\left( x,t\right) v_{1}\left( x\right) dx
\end{gather*}%
a.e. in $\left( 0,T\right) $, which is the weak form of (\ref{homogeniserat
problem}).

Now we continue by finding the local problem for each of the three cases.

\textit{Case 1: }$0<q<r<q+2$\textit{.} In (\ref{svag form}) we choose the
test function%
\begin{equation*}
v\left( x\right) c\left( t\right) =\varepsilon v_{1}\left( x\right)
v_{2}\left( \frac{x}{\varepsilon }\right) c_{1}\left( t\right) c_{2}\left( 
\frac{t}{\varepsilon ^{r}}\right)
\end{equation*}%
where $v_{1}\in D(\Omega )$, $v_{2}\in C_{\sharp }^{\infty }(Y)/%
\mathbb{R}
$, $c_{1}\in D(0,T)$ and $c_{2}\in C_{\sharp }^{\infty }(S)$ and obtain,
after differentiations%
\begin{gather*}
\int_{\Omega _{T}}-\varepsilon ^{q+1}u_{\varepsilon }\left( x,t\right)
v_{1}\left( x\right) v_{2}\left( \frac{x}{\varepsilon }\right) \partial
_{t}c_{1}\left( t\right) c_{2}\left( \frac{t}{\varepsilon ^{r}}\right) dxdt
\\
-\int_{\Omega _{T}}\varepsilon ^{q+1-r}u_{\varepsilon }\left( x,t\right)
v_{1}\left( x\right) v_{2}\left( \frac{x}{\varepsilon }\right) c_{1}\left(
t\right) \partial _{s}c_{2}\left( \frac{t}{\varepsilon ^{r}}\right) dxdt \\
+\int_{\Omega _{T}}\varepsilon a\left( \frac{x}{\varepsilon },\frac{t}{%
\varepsilon ^{r}}\right) \nabla u_{\varepsilon }\left( x,t\right) \cdot
\nabla v_{1}\left( x\right) v_{2}\left( \frac{x}{\varepsilon }\right)
c_{1}\left( t\right) c_{2}\left( \frac{t}{\varepsilon ^{r}}\right) dxdt \\
+\int_{\Omega _{T}}a\left( \frac{x}{\varepsilon },\frac{t}{\varepsilon ^{r}}%
\right) \nabla u_{\varepsilon }\left( x,t\right) v_{1}\left( x\right) \cdot
\nabla _{y}v_{2}\left( \frac{x}{\varepsilon }\right) c_{1}\left( t\right)
c_{2}\left( \frac{t}{\varepsilon ^{r}}\right) dxdt \\
=\int_{\Omega _{T}}\varepsilon f\left( x,t\right) v_{1}\left( x\right)
v_{2}\left( \frac{x}{\varepsilon }\right) c_{1}\left( t\right) c_{2}\left( 
\frac{t}{\varepsilon ^{r}}\right) dxdt\text{.}
\end{gather*}%
Letting $\varepsilon \rightarrow 0$, omitting terms that equal zero, we
obtain%
\begin{gather}
\lim_{\varepsilon \rightarrow 0}\left( \int_{\Omega _{T}}-\varepsilon
^{q+1-r}u_{\varepsilon }\left( x,t\right) v_{1}\left( x\right) v_{2}\left( 
\frac{x}{\varepsilon }\right) c_{1}\left( t\right) \partial _{s}c_{2}\left( 
\frac{t}{\varepsilon ^{r}}\right) dxdt\right.  \label{lokalt problem bevis}
\\
+\left. \int_{\Omega _{T}}a\left( \frac{x}{\varepsilon },\frac{t}{%
\varepsilon ^{r}}\right) \nabla u_{\varepsilon }\left( x,t\right)
v_{1}\left( x\right) \cdot \nabla _{y}v_{2}\left( \frac{x}{\varepsilon }%
\right) c_{1}\left( t\right) c_{2}\left( \frac{t}{\varepsilon ^{r}}\right)
dxdt\right) =0\text{.}  \notag
\end{gather}%
Using the fact that $r<q+2$ and by observing that $\varepsilon
^{q+1-r}=\varepsilon ^{q+2-r}\cdot \varepsilon ^{-1}$ the first term
vanishes due to Theorem \ref{Theorem-very weak} and then Theorem \ref%
{Theorem-gradsplit} gives%
\begin{gather*}
\int_{\Omega _{T}}\int_{\mathcal{Y}_{1,1}}a\left( y,s\right) \left( \nabla
u\left( x,t\right) +\nabla _{y}u_{1}\left( x,t,y,s\right) \right) \\
\times v_{1}\left( x\right) \cdot \nabla _{y}v_{2}\left( y\right)
c_{1}\left( t\right) c_{2}\left( s\right) dydsdxdt=0\text{.}
\end{gather*}%
The Variational lemma yields%
\begin{equation*}
\int_{Y}a\left( y,s\right) \left( \nabla u\left( x,t\right) +\nabla
_{y}u_{1}\left( x,t,y,s\right) \right) \cdot \nabla _{y}v_{2}\left( y\right)
dy=0
\end{equation*}%
a.e. in $\Omega _{T}\times S$ which is the weak form of (\ref{lokalt
problem1}).

\textit{Case 2: }$r=q+2$. Using the same test functions as in case 1 we
arrive at (\ref{lokalt problem bevis}). According to Theorems \ref%
{Theorem-gradsplit} and \ref{Theorem-very weak}, since $r=q+2$, we have%
\begin{gather*}
\int_{\Omega _{T}}\int_{\mathcal{Y}_{1,1}}-u_{1}(x,t,y,s)v_{1}\left(
x\right) v_{2}\left( y\right) c_{1}(t)\partial _{s}c_{2}(s)dydsdxdt \\
+\int_{\Omega _{T}}\int_{\mathcal{Y}_{1,1}}a(y,s)\left( \nabla u\left(
x,t\right) +\nabla _{y}u_{1}\left( x,t,y,s\right) \right) \\
\times v_{1}\left( x\right) \cdot \nabla _{y}v_{2}(y)c_{1}\left( t\right)
c_{2}(s)dyds=0\text{.}
\end{gather*}%
By applying the Variational lemma we get%
\begin{gather*}
\int_{\mathcal{Y}_{1,1}}-u_{1}(x,t,y,s)v_{2}(y)\partial _{s}c_{2}(s)dyds \\
+\int_{\mathcal{Y}_{1,1}}a(y,s)\left( \nabla u\left( x,t\right) +\nabla
_{y}u_{1}\left( x,t,y,s\right) \right) \cdot \nabla
_{y}v_{2}(y)c_{2}(s)dyds=0
\end{gather*}%
a.e. in $\Omega _{T}$, which is the weak form of (\ref{lokalt problem2}).

\textit{Case 3: }$r>q+2$\textit{.} Before deriving the local problem for
this case we establish the independence of $s$ in $u_{1}$. By choosing the
test function%
\begin{equation*}
v\left( x\right) c\left( t\right) =\varepsilon ^{r-q-1}v_{1}\left( x\right)
v_{2}\left( \frac{x}{\varepsilon }\right) c_{1}\left( t\right) c_{2}\left( 
\frac{t}{\varepsilon ^{r}}\right) \text{,}
\end{equation*}%
where $v_{1}\in D(\Omega )$, $v_{2}\in C_{\sharp }^{\infty }(Y)/%
\mathbb{R}
$, $c_{1}\in D(0,T)$ and $c_{2}\in C_{\sharp }^{\infty }(S)$, the weak form,
after differentiation, becomes%
\begin{gather*}
\int_{\Omega _{T}}-\varepsilon ^{r-1}u_{\varepsilon }\left( x,t\right)
v_{1}\left( x\right) v_{2}\left( \frac{x}{\varepsilon }\right) \partial
_{t}c_{1}\left( t\right) c_{2}\left( \frac{t}{\varepsilon ^{r}}\right) dxdt
\\
-\int_{\Omega _{T}}\varepsilon ^{-1}u_{\varepsilon }\left( x,t\right)
v_{1}\left( x\right) v_{2}\left( \frac{x}{\varepsilon }\right) c_{1}\left(
t\right) \partial _{s}c_{2}\left( \frac{t}{\varepsilon ^{r}}\right) dxdt \\
+\int_{\Omega _{T}}\varepsilon ^{r-q-1}a\left( \frac{x}{\varepsilon },\frac{t%
}{\varepsilon ^{r}}\right) \nabla u_{\varepsilon }\left( x,t\right) \cdot
\nabla v_{1}\left( x\right) v_{2}\left( \frac{x}{\varepsilon }\right)
c_{1}\left( t\right) c_{2}\left( \frac{t}{\varepsilon ^{r}}\right) dxdt \\
+\int_{\Omega _{T}}\varepsilon ^{r-q-2}a\left( \frac{x}{\varepsilon },\frac{t%
}{\varepsilon ^{r}}\right) \nabla u_{\varepsilon }\left( x,t\right)
v_{1}\left( x\right) \cdot \nabla _{y}v_{2}\left( \frac{x}{\varepsilon }%
\right) c_{1}\left( t\right) c_{2}\left( \frac{t}{\varepsilon ^{r}}\right)
dxdt \\
=\int_{\Omega _{T}}\varepsilon ^{r-q-1}f\left( x,t\right) v_{1}\left(
x\right) v_{2}\left( \frac{x}{\varepsilon }\right) c_{1}\left( t\right)
c_{2}\left( \frac{t}{\varepsilon ^{r}}\right) dxdt\text{.}
\end{gather*}%
Since $r>q+2$ all terms but the second one vanish as $\varepsilon
\rightarrow 0$. Due to Theorem \ref{Theorem-very weak} we have%
\begin{equation*}
\int_{\Omega _{T}}\int_{\mathcal{Y}_{1,1}}-u_{1}\left( x,t,y,s\right)
v_{1}\left( x\right) v_{2}\left( y\right) c_{1}\left( t\right) \partial
_{s}c_{2}\left( s\right) dydsdxdt=0
\end{equation*}%
and applying the Variational lemma we get%
\begin{equation*}
\int_{S}-u_{1}\left( x,t,y,s\right) \partial _{s}c_{2}\left( s\right) ds=0
\end{equation*}%
a.e. in $\Omega _{T}\times Y$, which implies that $u_{1}$ is independent of $%
s$. Now, to find the local problem, we choose the test function%
\begin{equation*}
v\left( x\right) c\left( t\right) =\varepsilon v_{1}\left( x\right)
v_{2}\left( \frac{x}{\varepsilon }\right) c_{1}\left( t\right) \text{,}
\end{equation*}%
where $v_{1}\in D(\Omega )$, $v_{2}\in C_{\sharp }^{\infty }(Y)/%
\mathbb{R}
$ and $c_{1}\in D(0,T)$. Carrying out differentiations, the weak form (\ref%
{svag form}) becomes%
\begin{gather*}
\int_{\Omega _{T}}-\varepsilon ^{q+1}u_{\varepsilon }\left( x,t\right)
v_{1}\left( x\right) v_{2}\left( \frac{x}{\varepsilon }\right) \partial
_{t}c_{1}\left( t\right) dxdt \\
+\int_{\Omega _{T}}\varepsilon a\left( \frac{x}{\varepsilon },\frac{t}{%
\varepsilon ^{r}}\right) \nabla u_{\varepsilon }\left( x,t\right) \cdot
\nabla v_{1}\left( x\right) v_{2}\left( \frac{x}{\varepsilon }\right)
c_{1}\left( t\right) dxdt \\
+\int_{\Omega _{T}}a\left( \frac{x}{\varepsilon },\frac{t}{\varepsilon ^{r}}%
\right) \nabla u_{\varepsilon }\left( x,t\right) v_{1}\left( x\right) \cdot
\nabla _{y}v_{2}\left( \frac{x}{\varepsilon }\right) c_{1}\left( t\right)
dxdt \\
=\int_{\Omega _{T}}\varepsilon f\left( x,t\right) v_{1}\left( x\right)
v_{2}\left( \frac{x}{\varepsilon }\right) c_{1}\left( t\right) dxdt\text{.}
\end{gather*}%
Theorem \ref{Theorem-gradsplit} and the fact that $u_{1}$ is independent of $%
s$ gives%
\begin{equation*}
\int_{\Omega _{T}}\int_{\mathcal{Y}_{1,1}}a\left( y,s\right) \left( \nabla
u\left( x,t\right) +\nabla _{y}u_{1}\left( x,t,y\right) \right) v_{1}\left(
x\right) \cdot \nabla _{y}v_{2}\left( y\right) c_{1}\left( t\right)
dydsdxdt=0
\end{equation*}%
and by the Variational lemma we have%
\begin{equation*}
\int_{Y}\left( \int_{S}a\left( y,s\right) ds\right) \left( \nabla u\left(
x,t\right) +\nabla _{y}u_{1}\left( x,t,y\right) \right) \cdot \nabla
_{y}v_{2}\left( y\right) dy=0
\end{equation*}%
a.e. in $\Omega _{T}$, which is the weak form of (\ref{lokalt problem3}).
\end{proof}

\end{document}